\documentclass[10pt]{amsart}
\usepackage[T1]{fontenc}
\usepackage{palatino}

\usepackage[all]{xy}                        %
\CompileMatrices                            
\UseTips                                    

\input xypic
\usepackage[bookmarks=true]{hyperref}       

\usepackage{amssymb,latexsym,amsmath,amscd}
\usepackage{mathtools}
\usepackage{commath}
\usepackage{xspace}
\usepackage{enumerate}
\usepackage{graphicx}
\usepackage{vmargin}
\usepackage{todonotes}
\usepackage{doi}

\usepackage{ulem}

\usepackage{dcpic,pictexwd}
\usepackage{tcolorbox}
\usepackage{subcaption}

\theoremstyle{plain}
\newtheorem{theorem}{Theorem}[section]
\newtheorem*{theorem*}{Theorem}
\newtheorem{proposition}[theorem]{Proposition}
\newtheorem{corollary}[theorem]{Corollary}
\newtheorem{lemma}[theorem]{Lemma}

\theoremstyle{definition}
\newtheorem{definition}[theorem]{Definition}

\newtheorem{notation}[theorem]{Notation}
\newtheorem{assumption}[theorem]{Assumption}

\newtheorem{remark}[theorem]{Remark}

\newcommand{\enm}[1]{\ensuremath{#1}}          %

\renewcommand{\bar}[1]{\overline{#1}}

\newcommand{\CC}{\enm{\mathbb{C}}}

\renewcommand{\imath}{\mathrm{i}}

\renewcommand{\phi}{\Phi}
\renewcommand{\theta}{\vartheta}
\renewcommand{\epsilon}{\varepsilon}

\newcommand{\Jint}{\mathcal{J}\mspace{-20mu}\int  }

\def\8{\infty}

\def\ge{\geqslant}

\def\Re{\mathop{\mathrm{Re}}}
\def\Im{\mathop{\mathrm{Im}}}

\usepackage{color}

\begin{document}
\parskip1pt

\title[$q$-Fock Space of $q$-Analytic Functions]{$q$-Fock Space of $q$-Analytic Functions and its realization in $L^{2}(\mathbb{C}; e^{-z\bar z} \,\mathrm{d}x\,\mathrm{d}y)$}

\author[A. Altavilla]{Amedeo Altavilla}\address{Dipartimento di Matematica,
  Universit\`a degli Studi di Bari Aldo Moro, via Edoardo Orabona, 4, 70125,
  Bari, Italia}\email{amedeo.altavilla@uniba.it}

\author[S. Bernstein]{Swanhild Bernstein}\address{
Fakult\"at f\"ur Mathematik und Informatik, Institut f\"ur Angewandte Analysis,
TU Bergakademie Freiberg, Pr\"uferstra\ss e, 9, 09599, Freiberg, Germany}
\email{Swanhild.Bernstein@math.tu-freiberg.de }

\author[M. L. Zimmermann]{Martha Lina Zimmermann}\address{
Fakult\"at f\"ur Mathematik und Informatik, Institut f\"ur Angewandte Analysis,
TU Bergakademie Freiberg, Pr\"uferstra\ss e, 9, 09599, Freiberg, Germany}\email{martha-lina.zimmermann@math.tu-freiberg.de}

\thanks{Amedeo Altavilla was partially supported by PRIN 2022MWPMAB - ``Interactions between Geometric Structures and Function Theories'' and by GNSAGA of INdAM}

\date{\today}

\subjclass[2020]{Primary: 46E22, 47B32; Secondary: 33D45, 81S05}
\keywords{q-Fock space; q-analyticity; Jackson calculus;  Hermite polynomials; q-special functions; reproducing kernels; creation–annihilation operators; q-Bargmann transform}

\begin{abstract}
We introduce a $q$-deformation of the Fock space of holomorphic functions on $\mathbb{C}$, based on a geometric definition of $q$-analyticity. This definition is inspired by 
a standard construction in complex differential geometry. Within this framework, we define $q$-analytic monomials $z_q^n$ and construct the associated $q$-Fock space as a Hilbert space with orthonormal basis $\{z_q^n/\sqrt{[n]_q!]}\}_{n\ge 0}$. The reproducing kernel  of this space is computed explicitly, and $q$-position and $q$-momentum operators are introduced, satisfying $q$-deformed commutation relations. We show that the $q$-monomials $z_q^n$ can be expanded in terms of complex Hermite polynomials, thereby providing a realization of the $q$-Fock space as a subspace of $L^2(\mathbb{C}; e^{-|z|^2}\,\mathrm{d}x\,\mathrm{d}y)$. Finally, we define a $q$-Bargmann transform that maps suitable $q$-Hermite functions into our $q$-Fock space and acts as a unitary isomorphism. Our construction offers a geometric and analytic approach to $q$-function theory, complementing recent operator-theoretic models.\end{abstract}

\maketitle
\setcounter{tocdepth}{1} 

\section{Introduction}
The classical Fock space of holomorphic functions on the complex plane $\CC$, equipped with a Gaussian weight, plays a fundamental role in quantum mechanics, signal processing, and complex analysis. It provides a natural setting for the Bargmann transform and coherent states, allowing the interplay of analysis, geometry, and operator theory~\cite{Barg1961}.

In recent years, several attempts have been made to define $q$-analogues of analytic structures, motivated by quantum algebra, combinatorics, and $q$-special functions (see~\cite{Ernst2012, Gasper2009, Kac2002}). The notions of $q$-derivative, Jackson integral, and $q$-exponential have led to a renewed interest in the development of a deformation of classical analysis, known as $q$-calculus. Recently, observations regarding $q$-deformations in different contexts such as Clifford analysis and quaternionic analysis were made (see~\cite{ACKS,BAZ,GonzCe,BeZiSchn2022}).

In particular, in the recent work~\cite{ACKS}, the authors study a family of $q$-deformed reproducing kernel Hilbert spaces $\mathcal{H}^2_q$, which interpolate between the classical Hardy and Fock spaces, via kernels of the form
\[
K_q(z, w) = \sum_{n=0}^\infty \frac{(z w)^n}{[n]_q!}.
\]
Their construction focuses on operator-theoretic properties, including the adjointness relations between $q$-difference operators and multiplication, and provides structural identities such as
\[
R_q M_z - q M_z R_q = I,
\]
which characterize these spaces.

Turning back to holomorphicity, it is well known that holomorphic functions can be defined in several equivalent ways. However, when one tries to generalize this concept based on just one of these definitions, it often leads to the loss of some of the other equivalent characterizations. This issue becomes particularly evident in quaternionic (or hypercomplex) analysis (see, for example,~\cite{Sudbery1979}). A similar pattern can be observed in the context of $q$-calculus. To gain a deeper understanding of this subject, we chose to explore it from a more geometric perspective, i.e., by assuming a sort of complex differentiability.

In fact, in the standard setting, a function $ f : U \subseteq \mathbb{C} \to \mathbb{C} $ is \textit{holomorphic} if and only if its differential $ \mathrm{d}f $ commutes with the action of multiplication by $ i $, meaning $ \mathrm{d}f(i\mathbf{v}) = i\,\mathrm{d}f(\mathbf{v}) $ for all $ \mathbf{v} \in \mathbb{R}^2 $. This formulation emphasizes holomorphy as a condition not merely on partial derivatives but on the geometric structure of the differential as a linear map preserving the complex structure.

In our setting, we will impose the condition
\[
\mathrm{d}_q f = \left( M_q^y D_z f \right) \mathrm{d}_q z,
\]
where $M_q^y$ is a suitable dilation operator.

Building upon this idea, we consider a $q$-deformation of the classical monomials $z^n$, denoted by $z_q^n$, which span a space of $q$-analytic functions. These functions form a basis of a $q$-Fock space defined by
\[
\mathcal{F}_q(\mathbb{C}) = \left\{ f(z) = \sum_{n=0}^\infty a_n z_q^n \;\middle|\; \sum_{n=0}^\infty |a_n|^2 [n]_q! < \infty \right\}.
\]
This space admits a natural inner product and a reproducing kernel structure analogous to the classical setting, yet deeply influenced by the $q$-deformation.

We further study $q$-position and $q$-momentum operators on $\mathcal{F}_q(\mathbb{C})$ and prove that they satisfy $q$-commutation relations deforming the canonical ones. Moreover, although the monomials $z_q^n$ are not holomorphic in the classical sense, each $z_q^n$ admits an expansion in complex Hermite polynomials $H_{p,r}(z,\bar z)$; this provides a concrete embedding (the \textit{realization} in the title) of the $q$-Fock space into $L^2(\mathbb{C}; e^{-|z|^2}\,\mathrm{d}x\,\mathrm{d}y)$.

Finally, we construct a $q$-Bargmann transform and we show that such a map is unitary. A two-variable tensor version yields an explicit coherent-state reproducing kernel and identifies the family $\{z_q^k\,\bar z_q^{\,h}\}_{k,h\ge0}$ and its span inside $L^2(\mathbb{C}; e^{-|z|^2}\,\mathrm{d}x\,\mathrm{d}y)$.

Lastly, some few words on the relation between our paper and~\cite{ACKS}.
While this paper and~\cite{ACKS}, both deal with $q$-deformations of classical spaces and involve similar structures, our approach differs significantly in the starting idea and technical formulation. In particular, \cite{ACKS} builds on standard power series in $z$, while our work uses $q$-deformed monomials $z_q^n$ based on geometric $q$-differentiability.
Moreover, we realize our $q$-Fock space inside $L^2(\mathbb{C}, e^{-|z|^2}\,\mathrm{d}x\,\mathrm{d}y)$ via expansion in complex Hermite polynomials, whereas~\cite{ACKS} considers weighted measures supported on discrete radial levels.
Finally, we introduce a $q$-Bargmann transform which maps $q$-Hermite functions into the space of $q$-analytic functions, similarly to the classical transform but adapted to the non-holomorphic $q$-setting.

\medskip
The paper is organized as follows.
In Section~\ref{preliminaries} we review Jackson $q$–calculus and introduce the complex $q$–derivatives $D_z,D_{\bar z}$ together with the $q$–analytic monomials $z_q^n$. 
Then, in Section~\ref{secHerm} we recall the $q$–Hermite polynomials and discuss some of their features: the $q$–Hermite equation, creation/annihilation identities for the $q$–Hermite polynomials and orthogonality on the Jackson $L^2$–space on $[-\lambda,\lambda]$. 
In Section~\ref{q-fockspace} we define our $q$–Fock space $\mathcal F_q(\mathbb{C})$ from the family $\{z_q^n\}$, compute its reproducing kernel $K_q$, develop the position/momentum pair and the commutation rule, and realize $\mathcal{F}_q(\mathbb{C})$ inside $L^2(\mathbb{C};e^{-|z|^2}\,\mathrm{d}x\,\mathrm{d}y)$. In the last Section we  define the $q$–Bargmann transform $\mathcal{B}_{q}$, prove that $\mathcal{B}_{q}$ is unitary, and introduce the two–variable transform $\mathcal{B}_{q}^{(2)}$.

\section{Preliminaries on $q$-Calculus}\label{preliminaries}
Here we briefly recall the standard definitions in $q$-calculus (see \cite{Ernst2012,Jackson1909,Kac2002}).
We start with the following general assumption.

\begin{assumption}
From now until the end of the article, the number $q$ is chosen to be in the open interval $(0,1)\subset\mathbb{R}$. Although many of the statements in the article also hold for other values of $q$ (real or complex), this choice is made to ensure the good position of Jackson's integral.
\end{assumption}

\subsection*{$q$-Derivatives and $q$-Numbers}

\begin{definition}
Let $\Omega \subset \mathbb{R}$ be a set such that if $ x \in \Omega $, then $ qx \in \Omega $ as well. Let $ f : \mathbb{R} \to \mathbb{R} $ be any function. The \textit{$q$-derivative} (or \textit{$q$-difference operator}) of $f$ is defined by:
\[
D_x^{q} f(x) = \frac{f(qx) - f(x)}{(q - 1)x}.
\]
\end{definition}

As an example, let $ \alpha \in \mathbb{R} $, and consider the function defined on the whole real line $ f(x) = x^{\alpha} $. Then:
\[
D_x^{q} x^{\alpha} = \frac{(qx)^{\alpha} - x^{\alpha}}{(q - 1)x} = \frac{q^{\alpha} - 1}{q - 1} x^{\alpha - 1}.
\]
This motivates the following definition:

\[
[\alpha]_q := \frac{1 - q^{\alpha}}{1 - q},
\]
so that $ D_x^{q} x^{\alpha} = [\alpha]_q x^{\alpha - 1} $ for any $ \alpha \in \mathbb{R} $. In particular, for integer values $ n \in \mathbb{N} $, this reduces to:
\[
[n]_q = 1 + q + q^2 + \dots + q^{n-1},
\]
and the corresponding $q$-factorial is defined as
\[
[n]_q! = [n]_q [n-1]_q \cdots [1]_q, \quad [0]_q! := 1.
\]
Using this, the $q$-binomial coefficients are given by:
\[
\left[\begin{matrix}n\\k\end{matrix}\right]_q = \frac{[n]_q!}{[n-k]_q! [k]_q!}.
\]

An identity that will be useful in later sections is:
\[
[n]_q - [n+1]_q = -q^n.
\]

\subsection*{$q$-Differentials}
From the definition of the $q$-derivative, the corresponding $q$-differential can be naturally introduced.

\begin{definition}
Let $\Omega \subset \mathbb{R}$ be a set such that if $ x \in \Omega $, then $ qx \in \Omega $. Let $ f : \Omega \to \mathbb{R} $ be any function. The \textit{$q$-differential} of $f$ is defined by:
\[
\mathrm{d}_q f(x) = f(qx) - f(x) = \left( D_x f(x) \right)\, \mathrm{d}_q x,
\]
where $ \mathrm{d}_q x := (q - 1)x $.
\end{definition}

The operator $ f(x) \mapsto f(qx) $ is often called the dilation operator, denoted $ M_q^x $, and acts as 
$$M_q^x f(x) = f(qx).$$

\begin{remark}
In operator form, the dilation operator is sometimes denoted as $ M_q^x = q^{x \frac{d}{dx}} $, resulting in the representation:
\[
D_x^{q} = \frac{q^{x \frac{\mathrm{d}}{dx}} - 1}{(q - 1)x}.
\]
In fact, the exponential operator $ e^{a \frac{\mathrm{d}}{\mathrm{d}x}} $ operates on a function $ f(x) $ through its formal Taylor series expansion:
\[
e^{a \frac{\mathrm{d}}{\mathrm{d}x}} f(x) = \sum_{n=0}^{\infty} \frac{a^n}{n!} \frac{\mathrm{d}^n f}{\mathrm{d}x^n}(x) = f(x + a),
\]
which translates the function by $ a $.

Similarly, the operator $ q^{x \frac{\mathrm{d}}{\mathrm{d}x}} $ is defined by:
\[
q^{x \frac{\mathrm{d}}{\mathrm{d}x}} := e^{(\ln q) x \frac{\mathrm{d}}{\mathrm{d}x}} = \sum_{n=0}^{\infty} \frac{(\ln q)^n}{n!} \left( x \frac{\mathrm{d}}{\mathrm{d}x} \right)^n,
\]
and acts on functions by performing a dilation:
\[
q^{x \frac{\mathrm{d}}{\mathrm{d}x}} f(x) = f(qx).
\]
This is because the operator $ x \frac{\mathrm{d}}{\mathrm{d}x} $ acts diagonally on monomials $ x^n $, with:
\[
x \frac{\mathrm{d}}{\mathrm{d}x} (x^n) = n x^n,
\]
and hence:
\[
e^{(\ln q) x \frac{\mathrm{d}}{\mathrm{d}x}} x^n = q^n x^n = (qx)^n.
\]
\end{remark}

For any two functions $ f, g : \Omega \to \mathbb{R} $ and any scalar $ \lambda \in \mathbb{R} $, the operator $D_{x}$ is linear:
\[
D_x(\lambda f + g)(x) = \lambda D_x f(x) + D_x g(x).
\]
Moreover, the $q$-derivative satisfies two distinct versions of the Leibniz product rule:
\begin{align*}
D_x(fg)(x) &= (D_x f(x)) g(x) + f(qx) D_x g(x), \\
           &= (D_x f(x)) g(qx) + f(x) D_x g(x).
\end{align*}
However, a general chain rule does not hold in this framework.

\subsection{$q$-Exponential Function}
A natural extension of the classical exponential function within the framework of $q$-calculus is the Jackson exponential.

\begin{definition}
For $ x \in \left(-\frac{1}{1 - q}, \frac{1}{1 - q}\right) $, the \textit{Jackson $q$-exponential} function is defined as
\[
E_q(x) := \sum_{k=0}^{\infty} \frac{x^k}{[k]_q!}.
\]
\end{definition}

\begin{remark}
The series $ E_q(x) $ converges absolutely and uniformly for $ |x| < \frac{1}{1 - q} $ when $ 0 < q < 1 $, and it converges for all $ x \in \mathbb{R} $ when $ q > 1 $ (see~\cite[Section 6.8]{Ernst2012}).
\end{remark}

A second $q$-exponential function is defined via an inverse relationship:
\[
e_q(x) := E_q^{-1}(x) = \sum_{j=0}^{\infty} q^{\frac{1}{2}j(j-1)} \frac{x^j}{[j]_q!}.
\]
These functions satisfy the identity (see~\cite{CoSo2010})
\[
E_q(t)\, e_q(-t) = 1,
\]
and their $q$-derivatives are given by:
\[
D_x^{q} E_q(x) = E_q(x), \quad D_x^{q} e_q(x) = e_q(qx).
\]
It is shown in~\cite[Lemma 11]{CoSo2010} that, for $ 0 < q < 1 $ and any $ k \in \mathbb{N} $,
\begin{equation}\label{convexp}
e_q\left( \frac{q^{-k}}{q - 1} \right) = 0.
\end{equation}

\subsection{Jackson Integral}
Jackson introduced a notion of integration that serves as an inverse operation to $q$-differentiation~\cite{Jackson1910}.
\begin{definition}
For $ 0 < q < 1 $, the \textit{Jackson $q$-integral} over the interval $ [0, a] $ is defined as:
\[
\Jint_0^a f(t)\, \mathrm{d}_q t = (1 - q) a \sum_{k=0}^{\infty} f(a q^k) q^k.
\]
\end{definition}
Integration over a general interval $ [a, b] $ is then defined by:
\[
\Jint_a^b f(t)\, \mathrm{d}_q t := \Jint_0^b f(t)\, \mathrm{d}_q t - \Jint_0^a f(t)\, \mathrm{d}_q t,
\]
and satisfies the $q$-analog of the Fundamental Theorem of Calculus:
\[
\Jint_a^b D_x f(x)\, \mathrm{d}_q x = f(b) - f(a).
\]

\begin{remark}
The assumption $ 0 < q < 1 $ ensures the convergence of the infinite sum defining the Jackson integral. For $ q > 1 $, this sum may diverge.
\end{remark}

\subsection{$q$-Analytic Functions}
We introduce the concept of $q$-analyticity, drawing inspiration from~\cite{Pashaev2014,Turner2016}.
Prior to this definition, we establish a family of sets within which this notion becomes pertinent.

\begin{definition}
Let $\Omega$ be a subset of $\mathbb{R}^{2}\simeq\mathbb{C}$. We will say that $\Omega$ is a $(q,q^{-1})$\textit{-invariant} set if, for any $(x,y)\in\Omega$, then $\{(qx,y),(x,q^{-1}y)\}\subset \Omega$.
\end{definition}

The whole set $\mathbb{C}$ is clearly a $(q,q^{-1})$\textit{-invariant} set. Other basic examples are the half-plane $\mathbb{C}^{+}=\{z\in\mathbb{C}\,|\, \Im(z)>0\}$, as well as $\{z\in\mathbb{C}\,|\, \Im(z)>0,\,\Re(z)>0\}$.

\begin{remark}
We want to show how to construct non-trivial families of $(q,q^{-1})$-invariant set .
Consider $(x,y)\in\mathbb{R}^2\cong\mathbb{C}$. We denote $(x,y)$ as the generating point of the set $\Omega_{(x,y)}$. By induction, we establish that the points $(qx,y)$ and $(x,q^{-1}y)$ are also members of $\Omega_{(x,y)}$. Consequently, we conclude that $\Omega_{(x,y)}=\{ (q^{b_1}x, q^{-b_2}y), b_1, b_2 \in \mathbb{N}\cup\{(0,0)\}\}$. Therefore, the set $\Omega_{(x,y)}$ is a grid generated by the point $(x,y)$ and the parameter $q$. Evidently, $\Omega=\cup_{k=1}^n\Omega_{(x_k,y_k)}$ satisfies the aforementioned condition. 

In Figure~\ref{fig1}, the reader can observe a portion of a set generated from the set of points \[S=\{(9.8,1.2),(10,1.2),(10.2,1.2),(9.8,1),(10,1),(10.2,1),(9.8,0.8),(10,0.8),(10.2,0.8)\}.\]
These are the nine points located in the bottom right corner of the image.
\begin{figure}[h]
\centering
\includegraphics[width=0.7\textwidth]{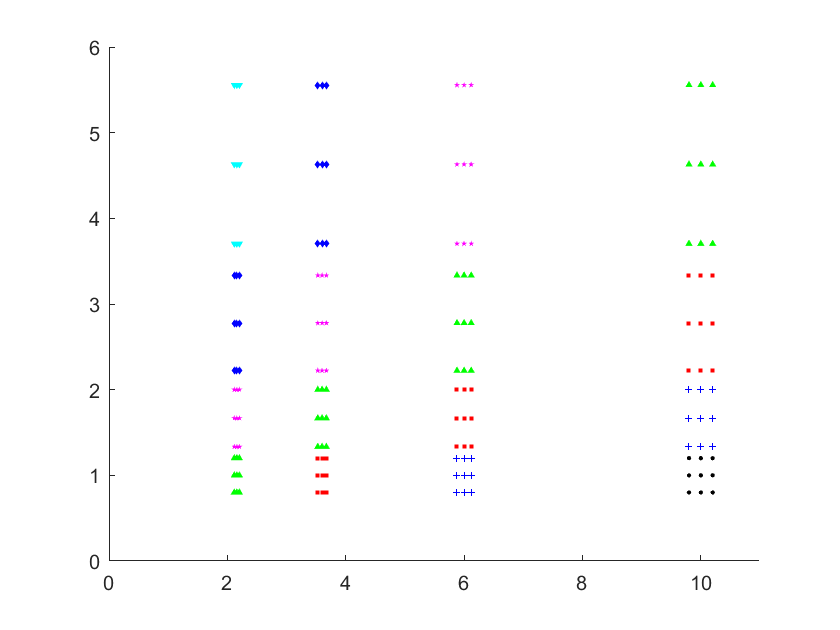}
\caption{The set generated by $S$ after $6$ iterations, where $q=0.6$.}
\label{fig1}
\end{figure}
\end{remark}

The $q$-differential of a function $f$ is defined as:
\[
\mathrm{d}_q f(x, y) = f(qx, qy) - f(x, y).
\]
This can be expressed in terms of $q$-derivatives as:
\[
\mathrm{d}_q f(x, y) = \left( M_q^y D_x^{q} f(x, y) \right)\, \mathrm{d}_q x + \left( D_y^{q} f(x, y) \right)\, \mathrm{d}_q y
\]
In complex coordinates, $z = x + iy$, $\bar{z} = x - iy$, we have:
\[
\mathrm{d}_q z = \mathrm{d}_q x + i \mathrm{d}_q y, \quad \mathrm{d}_q \bar{z} = \mathrm{d}_q x - i \mathrm{d}_q y.
\]
Therefore, the total $q$-differential becomes:
\[
\mathrm{d}_q f = \left( M_q^y D_z f \right)\, \mathrm{d}_q z + \left( M_q^y D_{\bar{z}} f \right)\, \mathrm{d}_q \bar{z},
\]
where the complex $q$-derivatives are defined as:
\[
D_z = \frac{1}{2} \left( D_x^{q} - i M_{1/q}^y D_y^{q} \right) = \frac{1}{2} \left( D_x^{q} - i D_y^{1/q} \right), \quad D_{\bar{z}} = \frac{1}{2} \left( D_x^{q} + i M_{1/q}^y D_y^{q} \right) = \frac{1}{2} \left( D_x^{q} + i D_y^{1/q} \right)
\]
and $D_{\chi}^{p}$ denotes the $p$-derivative with respect to the variable $\chi$.

\begin{remark}
Notice that, in order for a function $f:\Omega\to\mathbb{C}$ to admit $D_{z}$ and $D_{\bar z}$ derivatives we must ask $\Omega$ to be a $(q,q^{-1})$-invariant set.
\end{remark}

\begin{definition}
Let $\Omega\subset\mathbb{C}$ be a $(q,q^{-1})$-invariant set.
A function $ f : \Omega \to \mathbb{C} $ is said to be \textit{$q$-analytic} if
\[
D_{\bar{z}} f = 0.
\]
\end{definition}
Therefore, if $f$ is $q$-analytic, the $q$-differential simplifies to:
\[
\mathrm{d}_q f = \left( M_q^y D_z f \right)\, \mathrm{d}_q z.
\]
In the classical limit $ q \to 1 $, this reduces to the standard Cauchy-Riemann condition:
\[
\frac{\partial f}{\partial \bar{z}} = 0.
\]

\begin{remark}
In the classical holomorphic setting, a holomorphic function is uniquely determined up to a scalar. Conversely, a $q$-analytic function is uniquely determined up to a $(q,q^{-1})$-periodic function $A_{q}$, i.e., a function such that for any $(x,y)\in\Omega$, we have $A_{q}(x,y)=A_{q}(qx,y)=A_{q}(x,q^{-1}y)$.
These equalities imply that $A_{q}$ is constant on any grid $\Omega_{(x,y)}$ contained in $\Omega$.
\end{remark}

\begin{remark}
Analogously to the classical case, a function $ f : \Omega \to \mathbb{C} $ on a $(q,q^{-1})$-invariant set $\Omega$ is said to be \textit{$q$-anti-analytic} if it satisfies
\[
D_z f = 0.
\]
\end{remark}

\subsection{$q$-Analytic Polynomials and Power Series}
As discussed in~\cite{Am}, while the kernel of $D_{\bar{z}}$ contains unconventional functions, there exists a class of polynomials and power series that closely resemble conventional holomorphic functions.

\begin{definition}
Given $z = x + iy \in \mathbb{C}$ and $n \geq 2$, the \textit{$q$-analytic monomials} are defined recursively as follows:
\[
z_q^2 = (x + iy)(x + iqy), \quad z_q^n = z_q^{n-1}(x + iq^{n-1}y).
\]

\end{definition}

\begin{remark}
These monomials satisfy $ D_{\bar{z}} z_q^n = 0 $, and moreover:
\[
D_z z_q^n = [n]_q z_q^{n-1}.
\]
\end{remark}

Consequently, any linear combination of $z_q^n$ is $q$-analytic.
Analogously, one can define the $q$-anti-analytic monomials $\overline{z}_q^n$. (Clearly, $\overline{z}_q^n = \overline{z_q^n}$.) 

Note that, while $q$-analytic functions are not analytic functions (not even continuous functions~\cite{Am,BAZ}), the functions $z_{q}^{n}$ are generalized analytic functions~\cite{Pashaev2014} in the sense of Vekua~\cite{Vekua1962}.

\begin{remark}\label{remnorm}
For any positive integer $n$, it holds that $x^2 + q^{2n}y^2 \leq x^2 + y^2$, which implies:
\[
|z_q^n| \leq |z^n|.
\]
Consequently, if a function $f : \mathbb{D}_r \to \mathbb{C}$ is holomorphic with power series expansion $f(z) = \sum_{k=0}^\infty a_k z^k$, then the $q$-power series
\[
\sum_{k=0}^\infty a_k z_q^k
\]
is well-defined and $q$-analytic in any $(q,q^{-1})$-invariant subset contained in the same disc $\mathbb{D}_r$. In particular, if $f$ is an entire function, then $\sum_{k=0}^\infty a_k z_q^k$ is
well-defined and $q$-analytic in the whole space $\mathbb{C}$.
\end{remark}

\section{A Family of $q$-Deformed Hermite Polynomials}\label{secHerm}
In this section, we recall a family of $q$-deformed Hermite polynomials introduced in~\cite{Exton1980}, and later studied in a different context in~\cite{CoSo2010}. These polynomials will play a key role in the subsequent definition of the $q$-Bargmann transform.

\subsection{The $q$-Analog of the Gamma Function}

We begin by introducing a $q$-analogue of the Gamma function. For $q < 1$, it is defined by:
\[
\Gamma_q(t) = \frac{\prod_{k=1}^{\infty}(1 - q^k)}{\prod_{k=0}^{\infty}(1 - q^{t+k})}(1 - q)^{1 - t},
\]
and it satisfies the functional equation:
\[
\Gamma_q(t + 1) = [t]_q \Gamma_q(t),
\]
see~\cite{Gasper2009}.
Moreover, $\Gamma_q$ admits the following $q$-integral representation:
\[
\Gamma_q(z) = \Jint_0^{\frac{1}{1 - q}}  \, t^{z - 1} e_q(-q t)\,\mathrm{d}_q t,
\]
see~\cite{Exton1983, Floreanini1991}.
For $q < 1$, and setting $ \lambda = \sqrt{\frac{1}{1 - q^2}} $, the following identity holds:
\[
\Jint_{-\lambda}^{\lambda} t^{\nu - 1} e_{q^{2}}(-t^2) \,\mathrm{d}_q t
= \frac{2}{q + 1} q^{\nu} \Gamma_{q^{2}}\left( \frac{\nu}{2} \right).
\]

\subsection{The $q$-Hermite Equation}
Following~\cite{CoSo2010}, we consider the $q$-Hermite equation originally studied by Exton~\cite{Exton1980}.

\begin{definition}
The \textit{$q$-Hermite polynomial} $H_k^q$ is defined as a polynomial of the form
\[
H_k^q(t) = \sum_{j=0}^{\lfloor k/2 \rfloor} a_k^j t^{k - 2j},
\]
with initial coefficient $ a_k^0 = (q+1)^k $. These polynomials are eigenfunctions of the $q$-Hermite differential equation:
\[
\left( \left(D_t^q\right)^2 - (q+1)t D_t^q \right)f(t) = -(q+1)\lambda f(qt).
\]
\end{definition}
The corresponding eigenvalues are given by:
\[
\lambda_k = [k]_q q^{-k}.
\]
An explicit expression for $ H_k^q $ is:
\[
H_k^q(t) = \sum_{j=0}^{\lfloor k/2 \rfloor} (q+1)^{k-j} \frac{[k]_q!}{[k - 2j]_q! [-2j]_q [-2j+2]_q \cdots [-2]_q} t^{k - 2j}.
\]
As usual, in the classical limit $ q \to 1 $, these reduce to the standard Hermite polynomials:
\[
H_k(t) = \sum_{j=0}^{\lfloor k/2 \rfloor} (-1)^j 2^k \frac{k!}{(k - 2j)! j!} t^{k - 2j}.
\]

\subsection{Recurrence Relation, Annihilation and Creation Operators}
We now present the recurrence relation and annihilation operator satisfied by the $q$-Hermite polynomials. These can be derived by analyzing the structure of their expansion coefficients.

\begin{theorem}[Theorem 6 in~\cite{CoSo2010}]
The $q$-Hermite polynomials satisfy the recurrence relation:
\[
H_{k+1}^q = (q+1)tH_k^q - (q+1)[k]_q q^k H_{k-1}^q \quad \text{for } k > 0.
\]

Moreover, the $q$-derivative acts as an annihilation operator:
\[
D_t^q H_k^q(t) = (q+1)[k]_q H_{k-1}^q(t).
\]
\end{theorem}

In the classical case, the creation operator can be derived from either the recurrence relation or directly from the Hermite differential equation. However, in the $q$-deformed setting, these two approaches yield distinct creation operators.

For the $q$-Hermite polynomials $ H_k^q $, the following identities hold for $ k > 0 $:
\[
H_k^q(t) = \left( (q+1)t - q^k D_t^q \right) H_{k-1}^q(t), \quad
H_k^q(qt) = q^k \left( (q+1)t - D_t^q \right) H_{k-1}^q(t).
\]

\subsection{Exponential Function and Orthogonality}
Recall the $q$-exponential function:
\[
e_{q^2}(u) = \sum_{j=0}^\infty \frac{q^{j(j-1)} u^j}{[j]_q!}.
\]

This function satisfies the identity:
\[
D_t^q \left[ e_{ q^2}(-t^2) \right] = -(q+1)t\, e_{q^2}(-t^2).
\]

Using this identity and the Leibniz rule, one obtains:
$$
H_k^q(qt)\, e_{q^2}(-q^2 t^2) = -q^k D_t^q \left[ H_{k-1}^q(t)\, e_{q^2}(-t^2) \right].
$$

Theorem 8 in~\cite{CoSo2010} establishes the orthogonality of the $q$-Hermite polynomials. Recall from Equation~\eqref{convexp} that the convergence radius of $ e_{q^2}(-t^2) $ is $ \lambda = \frac{1}{\sqrt{1 - q^2}} $.

If $ 0 < q < 1 $, the $q$-Hermite polynomials are orthogonal with respect to the inner product
\begin{equation}\label{Iprod}
\langle f, g \rangle = \Jint_{-\lambda}^{\lambda} f(t)\, g(t)\, e_{q^2}(-t^2)\, \mathrm{d}_q t,
\end{equation}
where $ \lambda^2 = \frac{1}{1 - q^2} $. In particular, we have:
\begin{equation}\label{ForLambda}
\Jint_{-\lambda}^{\lambda} H_k^q(t)\, H_l^q(t)\, e_{q^2}(-t^2)\, \mathrm{d}_q t = \delta_{kl}\, 2(q+1)^{k-1}\, q^{\frac{1}{2}(k+1)(k+2)}\, [k]_q!\, \Gamma_{q^2}\left( \frac{1}{2} \right) = \Lambda\, \delta_{kl},
\end{equation}
with
\begin{equation}\label{biglambda}
\Lambda = 2(q+1)^{k-1}\, q^{\frac{1}{2}(k+1)(k+2)}\, [k]_q!\, \Gamma_{q^2}\left( \frac{1}{2} \right).
\end{equation}

\begin{remark}
This inner product is positive definite only when the functions are restricted to the discrete set of points $ \{ \pm \lambda q^j \mid j \in \mathbb{N} \} $.
\end{remark}

With the inner product in Formula~\eqref{Iprod}, we define a weighted  $L_q^2$-space:
\begin{definition}
Let $\lambda=\frac{1}{\sqrt{1-q^2}}$. The space
$L^2_q(\mathbb{R},e_{q^2}(-t^2),\lambda)$ is the closure of
\[
\operatorname{span}\Big\{\tfrac{1}{\sqrt{\Lambda}}\,H_k^q(t):k\in\mathbb{N}\Big\}
\]
with respect to the inner product
\[
\langle f,g\rangle=\Jint_{-\lambda}^{\lambda} f(t)\,g(t)\,e_{q^2}(-t^2)\,\mathrm{d}_q t.
\]
\end{definition}

\begin{remark} This definition resembles the situation in $L^2(\mathbb{R}),$ where the Hermite functions are a basis.
\end{remark}

\section{A $q$-Fock Space of $q$-analytic Functions}\label{q-fockspace}

In this section, we will introduce and study a $q$-deformed version of the classical Fock space.
Inspired by the work~\cite{ADDM}, we will define all the necessary concepts in terms of the coefficients of the functions involved in the space itself. Later, we will explore a method to reconcile the discrete and continuous approaches.
In the classical case~\cite{ADDM}, the reproducing kernel of the Fock space typically follows from
$$\sum_{n=0}^{\infty} u_n(\xi)u_n(\bar{z}) = \exp(\xi\bar{z})
$$
with $u_n(z) = \frac{z^n}{\sqrt{n!}}$.

In \cite{Pashaev2014} the authors describe a $q$-wave function built from the functions
$$
\psi_n(z_{q}) = \frac{z_{q}^{n}}{\sqrt{[n]_q!}}=\frac{(x+iy)_q^n}{\sqrt{[n]_q!}}.
$$
Let us consider then $\sum_{n=0}^{\infty} \psi_n(\xi_{q})\psi_n(\bar{z}_{q})$. We obtain
\begin{align*}
\sum_{n=0}^{\infty} \psi_n(\xi_{q})\psi_n(\bar{z}_{q}) &= \sum_{n=0}^{\infty} \frac{(\xi_1+i\xi_2)_q^n}{\sqrt{[n]_q!}}\cdot \frac{(z_1-iz_2)_q^n}{\sqrt{[n]_q!}}\\
&=\sum_{n=0}^{\infty} \frac{(\xi_1+i\xi_2)_q^n(z_1-iz_2)_q^n}{[n]_q!}.
\end{align*}
It is not possible to combine this into a $q$-deformed exponential. However, for $q \rightarrow 1$, we recover $\exp(\xi\bar{z})$. We must also consider the convergence of this series. While defining the wave functions $\psi_n$ from the $q$-coherent states, the authors of~\cite{Pashaev2014} argue that we need to restrict to the disc $\vert z \vert^2 \leq \frac{1}{1-q}$ for $\vert q \vert < 1$. For $q > 1$, the $q$-exponential function (Jackson’s exponential) converges without restriction.
We can still determine whether this is a reproducing kernel of our $q$-Fock space. 

Set
\begin{equation*}
K_q(z_{q},w_{q}) = \sum_{n=0}^{\infty} \frac{z_q^n\bar{w}_q^n}{[n]_q!}.
\end{equation*}

We are now able to define our $q$-Fock space.

\begin{definition}
The \textit{$q$-Fock space} is defined as
\begin{equation*}
\mathcal{F}_q(\mathbb{C}) = \left\{ f(z_{q}) = \sum_{n=0}^{\infty} z_q^n a_n\, :\, (a_n)_{n\in\mathbb{N}} \subset \mathbb{C},\, \sum_{n=0}^{\infty} \vert a_n \vert^2 [n]_q! < \infty\right\},
\end{equation*}
where $\vert q \vert < 1$.
\end{definition}

This $q$-Fock space can be endowed with the following inner product and the associated norm:
\begin{equation}\label{eq:fischer}
\langle f,g\rangle_{\mathcal{F}_q(\mathbb{C})} = \sum_{n=0}^{\infty} [n]_q! a_n\bar{b_n},\qquad
\lVert f \rVert_{\mathcal{F}_q(\mathbb{C})}^2 = \sum_{n=0}^{\infty} [n]_q! \vert a_n\vert^2.
\end{equation}
Now let us go back to our function $K_q$.

\begin{remark}
The inner product \eqref{eq:fischer} on the $q$–Fock space can be realized as a ($q$-)Fischer inner product:
\begin{equation}\label{eq:qFischer}
\langle f,g\rangle_{\text{F}}
= f(D_z)\,\overline{g}(0)
= \sum_{n=0}^\infty \frac{1}{[n]_q!}\, D_z^{\,n}f(0)\,\overline{D_z^{\,n}g(0)}
= \sum_{n=0}^\infty [n]_q!\, a_n\,\overline{b_n},
\end{equation}
where $f(z)=\sum_{n\ge0} a_n z^n$, $g(z)=\sum_{n\ge0} b_n z^n$. In particular,
$f(D_z)\,\overline{f}(0)=\sum_{n\ge0}[n]_q!\,|a_n|^2>0$ for $f\neq0$.
\end{remark}

\begin{theorem}
The function $K_q(z_{q},w_{q}) = \sum_{n=0}^{\infty} \frac{z_q^n\bar{w}_q^n}{[n]_q!}$ is the reproducing kernel of the $q$-Fock space $\mathcal{F}_q(\mathbb{C})$.
\end{theorem}
\begin{proof}
Let $\alpha(w_{q}) = \frac{\bar{w}_q^n}{[n]_q!}$. Then we have:
$$K_q(z_{q},w_{q}) = \sum_{n=0}^{\infty} z_q^n\alpha(w_{q}).$$
Consider the norm:
$$\lVert K_q \rVert_{\mathcal{F}_q(\mathbb{C})} = \sum_{n=0}^{\infty} [n]_q! \vert \alpha(w_{q}) \vert^2 = \sum_{n=0}^{\infty} \vert w_q^n\vert \frac{1}{[n]_q!} \leq \sum_{n=0}^\infty \vert w^n\vert\frac{1}{[n]_q!} = \sum_{n=0}^\infty \vert w \vert^n\frac{1}{[n]_q!} <\infty$$
if $q\in (0,1)$ and $\vert w \vert <\frac{1}{1-q}$. 
The inequality follows from Remark~\ref{remnorm}.
Therefore, $K_q(z_{q},w_{q}) \in \mathcal{F}_q(\mathbb{C})$.
Next, we need to show that $e_w(f) = f(w_{q})$ is a continuous linear functional and $\vert e_w(f)\vert \leq \lVert f \rVert_{\mathcal{F}_q(\mathbb{C})} \lVert K_q \rVert_{\mathcal{F}_q(\mathbb{C})} \, \forall f\in \mathcal{F}_q(\mathbb{C})$. Consider an arbitrary sequence of coefficients $(a_n)_{n\in\mathbb{N}} \subset \mathbb{C}$.
\begin{align*}
\vert e_w(f) \vert = \vert f(w_{q}) \vert &\leq \sum_{n=0}^{\infty} \vert w_q^n\vert \cdot \vert a_n \vert= \sum_{n=0}^{\infty} \sqrt{[n]_q!}\frac{1}{\sqrt{[n]_q!}}\vert w_q^n\vert\cdot \vert a_n\vert\\
&\leq \left( \sum_{n=0}^{\infty} \frac{1}{[n]_q!}\vert a_n \vert^2\right)^{1/2}\left(\sum_{n=0}^{\infty} [n]_q! (\vert w_q^n \vert)^2\right)^{1/2}\leq \lVert f \rVert_{\mathcal{F}_q(\mathbb{C})} \lVert K_q \rVert_{\mathcal{F}_q(\mathbb{C})},
\end{align*}
where the last inequality follows, again, thanks to Remark~\ref{remnorm}.

Finally we just need to verify that $K_q(z_{q},w_{q}) = \sum_{n=0}^{\infty} \frac{z_q^n\bar{w}_q^n}{[n]_q!}$ is a reproducing kernel of $\mathcal{F}_q(\mathbb{C})$, i.e. $\langle f, K_q \rangle_{\mathcal{F}_q(\mathbb{C})} = f(w_{q})$.\\
Take $f(z_{q}) = \sum_{n=0}^{\infty} z_q^n a_n \in \mathcal{F}_q(\mathbb{C})$ with $(a_n)_{n\in\mathbb{N}} \subset\mathbb{C}$ and set $b_n := \frac{\bar{w}_q^n}{[n]_q!}$.
$$
\langle f,K_q \rangle_{\mathcal{F}_q(\mathbb{C})} = \sum_{n=0}^{\infty} [n]_q! a_n\bar{b_n}= \sum_{n=0}^{\infty} [n]_q! a_n \frac{w_q^n}{[n]_q!}= \sum_{n=0}^{\infty} a_n w_q^n = f(w_{q}).
$$
\end{proof}

Let us consider $q$-binomials $z_q^n$. Their inner product is then:
$$
\langle z_q^n,z_q^m\rangle_{\mathcal{F}_q(\mathbb{C})} = [n]_q!\delta_{n,m}
$$
as we only have $a_k = b_k = 1$ if $m=n$.

\begin{remark}
The kernel $ K_q(z_{q}, w_{q}) $ is positive definite and Hermitian symmetric. By the Moore--Aronszajn theorem~\cite{Moore-Aronszajn}, it defines a unique reproducing kernel Hilbert space. However, this space does not coincide exactly with our $q$-deformed Fock space $ \mathcal{F}_q(\mathbb{C}) $.

Indeed, while the $q$-complex binomials $ z_q^n $ form a linearly independent set that spans $ \mathcal{F}_q(\mathbb{C}) $, they are not orthogonal with respect to the standard inner product. This lack of orthogonality arises because the monomials $ z_q^n $ generally depend on both $ z $ and $ \bar{z} $, as will be further discussed in Section~\ref{secL2}.

For the same reason, the inner product in $ \mathcal{F}_q(\mathbb{C}) $ cannot be defined using an integral formulation. Explicit computations show that $ \langle z_q^n, z_q^m \rangle \neq 0 $ even when $ n \neq m $, whenever the product $ z_q^n \cdot \bar{z}_q^m $ retains mixed dependence on $ z $ and $ \bar{z} $.

Nevertheless, we will attempt to construct an explicit realization of the $q$-Fock space within a proper Hilbert space framework in Section~\ref{secL2}.
\end{remark}

\subsection{Position and Momentum Operators}

In this subsection, we introduce the $q$-deformed analogues of the position and momentum operators and
we compute their commutator, which deforms the identity and tends to it as 
$q\to1$.

We define the creation operator $ a $ and the annihilation operator $ a^{\dagger} $  based on the $q$-position operator $X_q = z \circ M_q^y$  and the $q$-momentum operator $P_q = - iD_z $  by:
\[
a  = \frac{\sqrt{2}}{2} (X_q+iP_q)  \qquad a^{\dagger} = \frac{\sqrt{2}}{2} (X_q - iP_q) ,
\]
where $ z = x + i y $, and we recall that $ M_q^y $ denotes the dilation operator acting on the $ y $-variable, i.e., $ M_q^y(f(x,y)) = f(x, qy) $.
Note that this action satisfies:
\[
(z \circ M_q^y)(z_q^n) = z_q^{n+1}.
\]

\begin{remark}
Since
$
a=\tfrac{\sqrt2}{2}(X_q+iP_q)$, and $a^\dagger=\tfrac{\sqrt2}{2}(X_q-iP_q)
$,
we have that
\[
\begin{aligned}
[a,a^\dagger]
&=\tfrac12\,[X_q+iP_q,\,X_q-iP_q]\\
&=\tfrac12\big([X_q,-iP_q]+[iP_q,X_q]\big)
= -\,i[X_q,P_q]\\
&= -\,i[X_q,-\,iD_z]
= [D_z,X_q].
\end{aligned}
\]
Hence, $[a,a^\dagger]=[D_z,X_q]$. In particular,
$-i[X_q,P_q]=[D_z,X_q]$.
\end{remark}

In the classical case, the commutator $ [a, a^{\dagger}] $ equals the identity operator. In our $q$-deformed setting, the commutator deforms accordingly and tends to the identity as $ q \to 1 $.

\begin{lemma}
The commutator of the creation and annihilation operators satisfy
\[
[a,a^\dagger](z_q^n)= q^n\,z_q^n \quad\text{for all }n\ge0.
\]
\end{lemma}

\begin{proof}
Using $[a,a^\dagger]=[D_z,X_q]$ and the rules above,
\[
\begin{aligned}
[D_z,X_q](z_q^n)
&= D_z\big((z\circ M_q^y)(z_q^n)\big) - (z\circ M_q^y)\big(D_z z_q^n\big) \\
&= D_z(z_q^{n+1}) - (z\circ M_q^y)([n]_q z_q^{n-1})
= [n+1]_q z_q^{n} - [n]_q z_q^{n}
= q^n z_q^{n},
\end{aligned}
\]
since $[n+1]_q-[n]_q=q^n$. 
\end{proof}

\begin{corollary}
For any positive integer $n$ we have that
\[
-i[X_q,P_q](z_q^n)=q^n\,z_q^n.
\]
\end{corollary}

\begin{remark}
Recall the Fischer inner product from Formula~\eqref{eq:qFischer}; then, for arbitrary polynomials
\[
R(z)=\sum_{m\ge0} r_m z_q^m,\qquad Q(z)=\sum_{n\ge0} q_n z_q^n,
\]
we have
\[
\begin{aligned}
\langle D_z R,\,Q\rangle_{\mathrm F}
&=\Big\langle \sum_{m\ge1} r_m [m]_q z_q^{m-1},\,\sum_{n\ge0} q_n z_q^n\Big\rangle_{\mathrm F}
=\sum_{m\ge1} r_m [m]_q \,\overline{q_{m-1}}\,\langle z_q^{m-1},z_q^{m-1}\rangle_{\mathrm F}\\
&=\sum_{m\ge1} r_m [m]_q [m-1]_q!\,\overline{q_{m-1}}
=\sum_{m\ge1} r_m [m]_q!\,\overline{q_{m-1}}.
\end{aligned}
\]
On the other hand,
\[
\begin{aligned}
\langle R,\,X_q Q\rangle_{\mathrm F}
&=\Big\langle \sum_{m\ge0} r_m z_q^m,\,\sum_{n\ge0} q_n z_q^{n+1}\Big\rangle_{\mathrm F}
=\sum_{n\ge0} \overline{q_n}\,\langle z_q^{n+1},z_q^{n+1}\rangle_{\mathrm F}\,r_{n+1}\\
&=\sum_{n\ge0} r_{n+1}[n+1]_q!\,\overline{q_n}
=\sum_{m\ge1} r_m [m]_q!\,\overline{q_{m-1}},
\end{aligned}
\]
which equals $\langle D_z R,\,Q\rangle_{\mathrm F}$. Hence $(D_z)^\ast=X_q$ and $(X_q)^\ast=D_z$
with respect to the Fischer inner product. (In contrast, the oscillator-style pair
$a=\tfrac{\sqrt2}{2}(X_q+iP_q)$, $a^\dagger=\tfrac{\sqrt2}{2}(X_q-iP_q)$ 
is {not} an adjoint pair for the Fischer product.)
\end{remark}

\subsection{Realization of $\mathcal{F}_q(\mathbb{C})$  in $L^{2}(\mathbb{C}; e^{-z\bar z} dxdy)$}\label{secL2}

The family of functions $\mathcal{F} := \{z_q^n\}_{n \in \mathbb{N}}$ is not composed of holomorphic functions. Indeed, for any $ n \in \mathbb{N} $, each $ z_q^n $ can be expressed as~\cite{Pashaev2014,Am}
\[
z_q^{n+1} = z\left(\frac{1+q}{2}z + \frac{1-q}{2}\bar{z}\right)
\left(\frac{1+q^2}{2}z + \frac{1-q^2}{2}\bar{z}\right)
\dots \left(\frac{1+q^n}{2}z + \frac{1-q^n}{2}\bar{z}\right).
\]
If we define $ A_i = \frac{1 + q^{i}}{2} $ and $ B_i = \frac{1 - q^{i}}{2} $, then
\[
z_q^{n+1}  = z(A_1 z + B_1 \bar{z}) (A_2 z + B_2 \bar{z}) \dots (A_n z + B_n \bar{z}),
\]
and hence,
$$
z_q^{n+1} = \sum_{i + j = n} C_{i,j} z^{i+1} \bar{z}^j  = z\sum_{i + j = n} C_{i,j} z^{i} \bar{z}^j , \quad 
C_{i,j} = \sum_{\substack{|S| = i \\ S \cup S' = \{1, \dots, n\} \\  S \cap S' = \emptyset}} A_S B_{S'},
$$
where, if $ S = \{k_1, \dots, k_i\} $ and $ S' = \{h_1, \dots, h_j\} $, then $ A_S = A_{k_1} \dots A_{k_i} $ and $ B_{S'} = B_{h_1} \dots B_{h_j} $.
{ Similarly, }
$${ \bar{z}_q^{n+1} = \bar{z_q^{n+1}} = \sum_{i + j = n} C_{i,j} \bar z^{i+1} {z}^j  = \bar{z} \sum_{i + j = n} C_{i,j} \bar{z}^{i} z^j  . }
$$
Since these polynomials are "unbalanced" in powers of $ z $ and $ \bar{z} $, we do not expect the $q$-Fock space to be realized as a concrete Hilbert space of holomorphic functions. However, any monomial $ z^N \bar{z}^M $ can be expressed as a linear combination of suitable complex Hermite polynomials~\cite{Ghanmi}.
Set $p,r\in\mathbb{N}$. We define the complex Hermite polynomials in $z$ and $\bar z$ as
\begin{eqnarray}\label{hermite}
H_{p,r}(z, \bar{z}) = p! r! \sum_{k=0}^{\min(p,r)} \frac{(-1)^k}{k!} \frac{z^{p-k}}{(p-k)!} \frac{\bar{z}^{r-k}}{(r-k)!}.
\end{eqnarray}
This family forms an orthogonal basis for the Hilbert space~\cite{Ghanmi}
\[
L^2(\mathbb{C}; e^{-z\bar z} \, dx\,dy):=\{f:\mathbb{C}\to\mathbb{C}\,|\,\int_{\mathbb{C}} ||f||^{2}\, e^{-z\bar z} \, \mathrm{d}x \, \mathrm{d}y<\infty\},
\]
with orthogonality relation:
\[
\int_{\mathbb{C}} H_{p,r}(z, \bar{z}) \, \overline{H_{m,n}(z, \bar{z})} \, e^{-z\bar z} \, \mathrm{d}x \, \mathrm{d}y = \pi \, \delta_{pm} \, \delta_{rn} \, p! \, r!.
\]

Unfortunately, the family $ \{z_q^n\}_{n \in \mathbb{N}} $ is not complete in $ L^2(\mathbb{C}; e^{-z\bar z} \, \mathrm{d}x\,\mathrm{d}y) $, since $ \bar{z} $ cannot be expressed as a linear combination of these functions. The same holds for $ \{\bar{z}_q^n\}_{n \in \mathbb{N}} $. To obtain a complete family, one must consider all products $ z_q^k \bar{z}_q^h $.

\begin{proposition}\label{propbasis}
The set $ \{z_q^k \bar{z}_q^h\}_{k,h \in \mathbb{N}} $ forms a basis of the Hilbert space $ L^2(\mathbb{C}; e^{-z\bar z} \, \mathrm{d}x\,\mathrm{d}y) $.
\end{proposition}

Before performing the proof of Proposition~\ref{propbasis}, we need to state a remark and a technical lemma.
\begin{remark}[Elliptic complex variables]\label{elliptic}
For our consideration, we need to take into account complex variables of the following type:
$w=x+ipy$, with $0<p<\infty$. 
These are also called \textit{elliptic complex numbers}\footnote{Elliptic complex numbers were defined in~\cite{harkin} but our approach is slightly different.}.
Therefore, by transforming into polar coordinates, we get:
\[w\bar w=x^{2}+p^{2}y^{2}=r_{p}^{2}.\]
Hence, $\frac{w}{|w|}=\frac{x}{r_{p}}+i\frac{py}{r_{p}}$ is a unit complex number wich we
can write as $\cos(\varphi)+i\sin(\varphi)$, where $\cos(\varphi)=\frac{x}{r_{p}}$
and $\sin(\varphi)=\frac{py}{r_{p}}$. Then we get a relation between our polar coordinates $(r_{p},\varphi)$ and $(x,y)$, where $r_{p}=\sqrt{x^{2}+p^{2}y^{2}}$ and $\frac{1}{p}\frac{x}{y}=\tan\varphi$.
\end{remark}

\begin{lemma}\label{lemmaelliptic}
Let $w$ denote the elliptic complex number $x+ipy$. Then, for any $n\in\mathbb{N}$,
the family $\{w^{j}\bar w^{k}\}_{j+k=n}$ is a linearly independent set. 
\end{lemma}

\begin{proof}
Let us consider the following linear combination:
\[\sum_{j+k=n}a_{j,k}w^{j}\bar w^{k}=0,\]
where $a_{i,j}\in\mathbb{C}$ for all $i+j=n$.
By means of polar coordinates (see Remark~\ref{elliptic}), we get 
\[\sum_{j+k=n}a_{j,k}r_{p}^{n}e^{i(j-k)\varphi}=0.\]
First of all we can drop $r_{p}^{n}$ since it is different from zero if $w\neq 0$.
Then, we may write
\[0=\sum_{j+k=n}a_{j,k}e^{i(j-k)\varphi}=\sum_{j=0}^{n}a_{j,n-j}e^{i\varphi(2j-n)}.\]
Let us now multiply both sides by $e^{in\varphi}$. Then, for any $\varphi$ we get
\begin{equation}\label{phi}
0=\sum_{j=0}^{n}a_{j,n-j}(e^{2i\varphi})^{j}.
\end{equation}

Let now $\varphi_{\ell}:=\frac{\pi\ell}{2(n+1)}$, for $\ell=0,\dots,n$. Then, for any
$\ell$, equation~\eqref{phi} is a linear equation in the $a_{j,n-j}$ variables.
This is a system of $n+1$ equations in $n+1$ variables and the matrix of the system is a Vandermonde matrix $V=V(\varphi_{0},\dots,\varphi_{n})$ whose determinant is non-zero.
This implies that $a_{j,k}=0$ for all $j+k=n$ and we are done.
\end{proof}

The proof of the previous Lemma was inspired by~\cite{mathexc}.

\begin{proof}[Proof of Proposition~\ref{propbasis}]
The space $ L^2(\mathbb{C}; e^{-z\bar z} \, \mathrm{d}x\,\mathrm{d}y) $ admits both the Hermite polynomials $ H_{p,r} $ (see~\ref{hermite}) and the monomials $ z^p \bar{z}^r $ as bases, the former orthogonal, the latter not. 
To prove that $ \{z_q^k \bar{z}_q^h\}_{k,h \in \mathbb{N}} $ is a basis, it suffices to note that each such element is a linear combination of $ z^p \bar{z}^r $, hence lies in the Hilbert space. 
It remains to show that, for each fixed $ N \in \mathbb{N} $, the subfamily $ \{z_q^k \bar{z}_q^h\}_{k + h = N} $ forms a basis for the $ \mathrm{span}\{z^p \bar{z}^r \mid p + r = N\} $.

The number of elements is the same in both sets, so it suffices to prove linear independence. 

For $ N = 1 $, we have $ z_q = z $ and $ \bar{z}_q = \bar{z} $, which are linearly independent. For $ N = 2 $,
\[
z_q^2 = z(A_1 z + B_1 \bar{z}) = A_1 z^2 + B_1 z \bar{z}, \quad
z_q \bar{z}_q = z \bar{z}, \quad
\bar{z}_q^2 = \bar{z}(A_1 \bar{z} + B_1 z) = B_1 z \bar{z} + A_1 \bar{z}^2.
\]
Suppose a linear combination $ D_{2,0} z_q^2 + D_{1,1} z_q \bar{z}_q + D_{0,2} \bar{z}_q^2 = 0 $. Since $ z^2 $ appears only in $ z_q^2 $, it follows that $ D_{2,0} = 0 $; similarly $ D_{0,2} = 0 $ from the term $ \bar{z}^2 $. Therefore, $ D_{1,1} = 0 $ as well.

We now pass to the case $N=3$ and then proceed for the general case.
We have
\[
\begin{cases}
z_q^3=z(A_1 z + B_1 \bar{z}) (A_2 z + B_2 \bar{z}) \\
z_q^2\bar z_{q}=z(A_1 z + B_1 \bar{z}) \bar z\\
z_q\bar z_q^2=z\bar z(B_1 z + A_1 \bar{z}) \\
\bar z_q^3=\bar z(B_1 z + A_1 \bar{z}) (B_2 z + A_2 \bar{z}).
\end{cases}
\]
Hence, if we consider a linear combination
\[D_{3,0}z_q^3+D_{2,1}z_q^2\bar z_{q}+D_{1,2}z_q\bar z_q^2+D_{0,3}\bar z_q^3=0,\]
then the only summands containing $z^{3}$ and $\bar z^{3}$ are the first and the last, respectively.
Therefore, $D_{3,0}=D_{0,3}=0$ and we are left with
\begin{align*}
0&=D_{2,1}z_q^2\bar z_{q}+D_{1,2}z_q\bar z_q^2=z\bar z(D_{2,1}(A_1 z + B_1 \bar{z})+D_{1,2}(B_1 z + A_1 \bar{z}))\\
&=z\bar z((D_{2,1}A_{1}+D_{1,2}B_{1})z+(D_{2,1}B_{1}+D_{1,2}A_{1})\bar z).
\end{align*}
Now, since $A_{1}^{2}-B_{1}^{2}\neq 0$, then we get that $D_{2,1}=D_{1,2}=0$.

 Consider now the family $ \{z_q^k \bar{z}_q^h\}_{k + h = N+1} $ and suppose:
\[
\sum_{k+h = N+1} D_{k,h} z_q^k \bar{z}_q^h = 0.
\]
As before, the only terms containing $ z^{N+1} $ and $ \bar{z}^{N+1} $ appear in $ z_q^{N+1} $ and $ \bar{z}_q^{N+1} $, respectively, so $ D_{N+1,0} = D_{0,N+1} = 0 $. What remains is:
\begin{equation}\label{sumC}
D_{N,1} z_q^N \bar{z}_q + D_{N-1,2} z_q^{N-1} \bar{z}_q^2 + \dots + D_{1,N} z_q \bar{z}_q^N = 0.
\end{equation}
Now, for any $ k \in \mathbb{N} $, we can express
\[
z_q^k = z_{q^0} z_{q^1} \dots z_{q^{k-1}}, \quad \text{with } z_{q^\ell} := x + i q^\ell y,
\]
so $ z_q^k = z \cdot (z_{q^1})_q^{k-1} $. Then equation~\eqref{sumC} becomes:
\begin{equation}\label{eqCN}
z \bar{z} \left( D_{N,1}(z_{q^1})_q^{N-1} + D_{N-1,2}(z_{q^1})_q^{N-2} \bar{z}_{q^1} + \dots + D_{1,N}(\bar{z}_{q^1})_q^{N-1} \right) = 0,
\end{equation}
where $ z_{q^1} = x + i q y $. 

We now repeat the previous argument where $z$ is replaced by the elliptic complex variable
$w=z_{q^{1}}$ (see Remark~\ref{elliptic}). Therefore, Formula~\eqref{eqCN} can be written as
 \begin{equation}\label{eqCN2}
z \bar{z} \left( D_{N,1}w_q^{N-1} + D_{N-1,2}w_q^{N-2} \bar w_{q} + \dots + D_{1,N}\bar w_q^{N-1} \right) = 0.
\end{equation}
Similarly as before we have $(\bar w_{q})^{M}=\overline{w_{q}^{M}}$ and,
\[
w_q^{M} = w\left(A_{1}w + B_{1}\bar{w}\right)
\left(A_{2}w + B_{2}\bar{w}\right)
\dots \left(A_{M-1}w + B_{M-1}\bar{w}\right).
\]
Hence, again the only instances of $w^{N-1}$ and of $\bar w^{N-1}$ are contained in the first and last summands of Equation~\eqref{eqCN2}, respectively. 
Therefore, thanks to Lemma~\ref{lemmaelliptic}, $D_{N,1}$ and
$D_{1,N}$ must be equal to zero.

Again, as before, we are allowed to collect $w\bar w$ from 
the left hand side of Formula~\eqref{eqCN2} and so we are left with 
$$z \bar{z} w\bar w\left( D_{N-1,2}w_q^{N-3} + \dots +D_{2,N-1}\bar w_{q}^{N-3} \right) = 0.
$$
We now start an iteration process in which at each step we define a new elliptic complex variable $z_{q^{\ell}}$, so then the first and last coefficients $D_{N-\ell,\ell+1}$ and
$D_{\ell+1,N-\ell}$ must vanish and we can factorize $z_{q^{\ell}}\bar z_{q^{\ell}}$ again.

Eventually we will end up in one of the two cases $N=2$ or $N=3$ discussed before, depending whether the starting power $N+1$ is even or odd. 
\end{proof}

\section{A $q$-Bargmann transform}

In this section we construct a $q$-Bargmann transform that identifies the
 Jackson $L^2_{q}$-space on $[-\lambda,\lambda]$ with the $q$–Fock space
endowed with the Fischer inner product. We introduce the kernel,
compute the image of the $q$–Hermite basis, and prove unitarity.

Fix $\lambda=\frac{1}{\sqrt{1-q^2}}$. Let $\{H_k^q\}_{k\ge0}$ be the $q$–Hermite polynomials discussed in Section~\ref{secHerm}
and define the $q$–Hermite functions
\[
\tilde{H}_k^q(t):=\frac{1}{\sqrt{\Lambda}}\,H_k^q(t)\,\sqrt{e_{q^2}(-t^2)}\,,
\]
where $\Lambda$ is given in Formula~\eqref{biglambda}.
By construction the family $\{\tilde{H}_k^q\}_{k\ge0}$ is orthonormal in
$L_q^2(\mathbb{R},\lambda)$ with respect to the Jackson inner product
\[
\langle f,g\rangle_{L^2_q}=\Jint_{-\lambda}^{\lambda} f(t)\,\overline{g(t)}\,\mathrm{d}_q t.
\]
On the $q$–Fock space $\mathcal{F}_q(\mathbb{C})$ we use the orthonormal basis
$\big\{\frac{z_q^n}{\sqrt{[n]_q!}}\big\}_{n\ge0}$.

\begin{definition}
Define the kernel
\begin{equation}\label{eq:Bkernel}
A(z_q,t):=\sum_{n=0}^{\infty}\frac{z_q^n}{\sqrt{[n]_q!}}\;\tilde{H}_n^q(t),
\end{equation}
and the linear map $\mathcal{B}_q:L_q^2(\mathbb{R},\lambda)\to\mathcal{F}_q(\mathbb{C})$ by
\begin{equation}\label{eq:Btransform}
(\mathcal{B}_qf)(z_q):=\Jint_{-\lambda}^{\lambda} f(t)\,A(z_q,t)\,\mathrm{d}_q t.
\end{equation}
\end{definition}

\begin{proposition}\label{prop:BasisImage}
For every $m\ge0$,
$$\mathcal{B}_q(\tilde{H}_m^q)(z_q)=\frac{z_q^m}{\sqrt{[m]_q!}}.
$$\end{proposition}

\begin{proof}
Using \eqref{eq:Bkernel}–\eqref{eq:Btransform} and orthonormality of $\{\tilde H_n^q\}$,
\[
\mathcal{B}_q(\tilde{H}_m^q)
=\Jint_{-\lambda}^{\lambda}\tilde{H}_m^q(t)\!\left(\sum_{n=0}^{\infty}
\frac{z_q^n}{\sqrt{[n]_q!}}\tilde{H}_n^q(t)\right)\!\mathrm{d}_q t.
\]
Now consider the sequence:
\[
\left(\varphi_N\right)_{N \in \mathbb{N}} := 
\left( \Jint_{-\lambda}^{\lambda} \tilde{H}_m^q(t) 
\sum_{n=0}^{N} \frac{z_q^n}{\sqrt{[n]_q!}} \tilde{H}_n^q(t) \, \mathrm{d}_q t \right)_{N\in\mathbb{N}}.
\]
This sequence satisfies:
\[
\varphi_N = 0 \quad \text{for } N < m, \quad 
\varphi_N = \frac{z_q^m}{\sqrt{[m]_q!}} \quad \text{for } N \geq m.
\]
Therefore, it becomes constant for \( N \geq m \) and converges to that value. Using the orthonormality relation from Equation~\eqref{ForLambda}, we conclude:
$$\mathcal{B}_q(\tilde{H}_m^q)(z_{q}) 
= \sum_{n=0}^{\infty} \frac{z_q^n}{\sqrt{[n]_q!}} 
\Jint_{-\lambda}^{\lambda} \tilde{H}_n^q(t) \tilde{H}_m^q(t) \, \mathrm{d}_q t  
= \sum_{n=0}^{\infty} \frac{z_q^n}{\sqrt{[n]_q!}} \delta_{n,m}  
= \frac{z_q^m}{\sqrt{[m]_q!}}.
$$\end{proof}

\begin{theorem}\label{thm:unitary-B}
The map $\mathcal{B}_q:L_q^2(\mathbb{R},\lambda)\to\mathcal{F}_q(\mathbb{C})$ is unitary.
\end{theorem}

\begin{proof}
By Proposition~\ref{prop:BasisImage}, $\mathcal{B}_q$ maps the orthonormal basis
$\{\tilde H_m^q\}_{m\ge0}$ of $L_q^2(\mathbb{R},\lambda)$ onto the orthonormal basis
$\big\{\frac{z_q^m}{\sqrt{[m]_q!}}\big\}_{m\ge0}$ of $\mathcal{F}_q(\mathbb{C})$ (with Fischer product).
Hence $\mathcal{B}_q$ is an isometry on the finite span of the basis and extends by continuity
to a surjective isometry between the completions, i.e., a unitary operator.
Equivalently,
\[
\big\langle \mathcal{B}_q\tilde H_m^q,\mathcal{B}_q\tilde H_n^q\big\rangle_{\text{F}}
=\delta_{mn}
=\big\langle \tilde H_m^q,\tilde H_n^q\big\rangle_{L^2_q},
\]
hence $\mathcal{B}_q$ is unitary.
\end{proof}

\subsection{Tensor Product of the Bargmann Transform and its conjugate}
The transform introduced by V.~Bargmann is bijective and isometric from $L^2(\mathbb{R})$ onto the classical Fock space; in particular the Hermite functions are mapped to the normalized analytic monomials, which form an orthonormal basis of entire functions. 
In our setting, while we found a $q$-Bargmann transform $\mathcal{B}_q$ mapping the $q$-Hermite basis to the powers of $z_q$ (Theorem~\ref{thm:unitary-B}), the family $\{z_q^k\}_{k\ge0}$ is not a basis of $L^2(\mathbb C; e^{-z\bar z}\,dx\,dy)$. 
Instead, Section~\ref{secL2} shows that the bidimensional family $\{z_q^k\,\bar z_q^h\}_{k,h\in\mathbb N}$ is complete (Proposition~\ref{propbasis}). 
This motivates a ``tensorial'' construction, reminiscent of the classical tensor product of Bargmann transforms.

\medskip

Define the {bidimensional $q$-Fock space}
\[
\mathcal{F}^{(2)}_q(\mathbb{C})\ :=\ \overline{\mathrm{span}}\left\{\frac{z_q^k\,\bar z_q^h}{\sqrt{[k]_q!\,[h]_q!}}:\ k,h\in\mathbb{N}\right\},
\qquad 
\big\langle z_q^k\bar z_q^h,\ z_q^m\bar z_q^n\big\rangle_{\mathcal{F}^{(2)}_q}
=[k]_q!\,[h]_q!\,\delta_{km}\delta_{hn}.
\]

We introduce the factorized kernel
\begin{equation}\label{eq:tensor-kernel}
A^{(2)}(z_q,\bar z_q;\,t,s)\ :=\ A(z_q,t)\,A(\bar z_q,s)
=\frac{1}{\Lambda}\sum_{k,h\ge0}\frac{z_q^k\,\bar z_q^h}{\sqrt{[k]_q!\,[h]_q!}}\ \widetilde H^q_k(t)\,\widetilde H^q_h(s).
\end{equation}
Since $\widetilde H^q_k$ are real–valued, $A^{(2)}(z_q,\bar z_q;\,t,s)$ is Hermitian-symmetric in $(z_q,\bar z_q)$.

\begin{notation}
We write $H\widehat\otimes K$ for the {Hilbert tensor product}, i.e., the completion of
the algebraic tensor product $H\otimes K$ with respect to the inner product
$\langle u_1\!\otimes\! v_1,\,u_2\!\otimes\! v_2\rangle=\langle u_1,u_2\rangle_H\,\langle v_1,v_2\rangle_K$.
In particular,
\[
L^2_q(\mathbb{R},\lambda)\,\widehat\otimes\,L^2_q(\mathbb{R},\lambda)
\cong L^2_q([-\lambda,\lambda]^2,\lambda^{\otimes 2}),
\]
and the orthonormal basis $\{\widetilde H^q_k\}_{k\ge0}$ induces
$\{\widetilde H^q_k\otimes \widetilde H^q_h\}_{k,h\ge0}$.
\end{notation}
\begin{definition}
Define
\[
\mathcal{B}_q^{(2)}:L^2_q(\mathbb{R},\lambda)\ \widehat\otimes\ L^2_q(\mathbb{R},\lambda)\longrightarrow \mathcal{F}^{(2)}_q(\mathbb{C}),
\qquad 
(\mathcal{B}_q^{(2)}f)(z_q,\bar z_q):=\Jint_{[-\lambda,\lambda]^2}\!\! f(t,s)\,A^{(2)}(z_q,\bar z_q;\,t,s)\,\mathrm{d}_qt\,\mathrm{d}_qs.
\]
\end{definition}

\begin{theorem}
The map $\mathcal{B}_q^{(2)}$ is unitary. Moreover,
\[
\mathcal{B}_q^{(2)}\!\big(\,\widetilde H^q_k\otimes \widetilde H^q_h\,\big)(z_q,\bar z_q)
=\frac{z_q^k\,\bar z_q^h}{\sqrt{[k]_q!\,[h]_q!}}\qquad (k,h\in\mathbb{N}).
\]
\end{theorem}

\begin{proof}
Using \eqref{eq:tensor-kernel} and orthonormality of $\{\widetilde H^q_n\}$ in $L^2_q(\mathbb{R},\lambda)$, we get
\[
\mathcal{B}_q^{(2)}(\widetilde H^q_k\!\otimes\!\widetilde H^q_h)
=\sum_{m,n\ge0}\frac{z_q^m\,\bar z_q^n}{\sqrt{[m]_q!\,[n]_q!}}
\left(\Jint_{-\lambda}^{\lambda}\!\widetilde H^q_k(t)\widetilde H^q_m(t)\,\mathrm{d}_qt\right)
\left(\Jint_{-\lambda}^{\lambda}\!\widetilde H^q_h(s)\widetilde H^q_n(s)\,\mathrm{d}_qs\right)
=\frac{z_q^k\,\bar z_q^h}{\sqrt{[k]_q!\,[h]_q!}}.
\]
Hence $\mathcal{B}_q^{(2)}$ maps the orthonormal tensor basis $\{\widetilde H^q_k\otimes \widetilde H^q_h\}$ onto the orthonormal basis 
$\{\frac{z_q^k\bar z_q^h}{\sqrt{[k]_q!\,[h]_q!}}\}$ of $\mathcal{F}^{(2)}_q(\mathbb{C})$, so it is an isometry on finite spans and extends by continuity to a surjective isometry between the completions.
\end{proof}

\begin{remark}
Setting
\[
\Phi_{z_q}(t):=A(z_q,t)=\sum_{n\ge0}\frac{z_q^n}{\sqrt{[n]_q!}}\ \widetilde H^q_n(t),
\]
the kernel factorizes as $A^{(2)}(z_q,\bar z_q;\,t,s)=\Phi_{z_q}(t)\,\Phi_{\bar z_q}(s)$. 
Thus $\{\Phi_{z_q}\}_{z\in\mathbb{C}}$ may be regarded as a family of $q$–coherent state vectors for the analysis operator $\mathcal{B}_q$, and $A^{(2)}$ provides the natural tensor product coherent kernel. 
\end{remark}

\end{document}